\documentclass[11pt,reqno]{amsproc}
\usepackage{amsmath, amsthm, amscd, amsfonts, amsthm, amssymb, graphicx}
\newtheorem{theorem}{\bf{Theorem}}[section] 
\newtheorem{lemma}[theorem]{\bf{Lemma}}     

\newtheorem{corollary}[theorem]{\bf{Corollary}}
\newtheorem{proposition}[theorem]{\bf{Proposition}}

\title[Weakly Nuclear Maps on Real $C^*$-algebras and Quasidiagonality of These Algebras]
{Weakly Nuclear Maps on Real $C^*$-algebras and Quasidiagonality of These Algebras} 

\author{Ali Ebadian}
\address{
Department of Mathematics, Faculty of Science, Urmia University, Urmia, Iran}
\email{ebadian.ali@gmail.com}
\author{Ali Jabbari}
\email{jabbari\underline{ }al@yahoo.com}
\address{Young Researchers and Elite Club, Ardabil, Iran}
\subjclass[2010]{46L05}
\keywords{Completely positive maps, Exact $C^*$-algebras, Quasidiagonal $C^*$-algebras, Real $C^*$-algebras, Tracial functionals,  Weakly nuclear maps}

\begin{document}
\maketitle

\begin{abstract}
In this paper, we show that a completely positive linear map is weakly nuclear if and only if its complexification is weakly nuclear. It is shown that a real $C^*$-algebra is exact if and only if its complexification is exact and similar case is provided for the quasidiaginality.
\end{abstract}


\section{Introduction} 
In general, a real $C^*$-algebra is a  real Banach $*$-algebra which satisfies  $\|a^*a\|=\|a\|^2$ and $1+a^*a$ is invertible in the unitization $A^\#$ for every $a\in A$. Let $A$ be a (complex) $C^*$-algebra and $\Phi$ be an involutory $*$-antiautomorphism of $A$. Then   $A_\Phi=\{a\in A: \ \Phi(a)=a^*\}$ is a real subalgebra of $A$  such that $A_\Phi\cap iA_\Phi=\{0\}$ and $A=A_\Phi+iA_\Phi$ and $A$ is called the complexification of $A_\Phi$, for more details related to this concept, we refer to \cite{pa, sta, sto}. We call $A_\Phi$ is a Real $C^*$-algebra which is different from the real $C^*$-algebras that we have defined before. From now on, by a real $C^*$-algebra, we mean a $C^*$-algebra such as $A_\Phi$.  Let $A$ and $B$ be two $C^*$-algebras and let $\Phi$ and $\Psi$ be two involutory $*$-antiautomorphisms of $A$ and $B$, respectively. Assume that $\varphi:A_\Phi\longrightarrow A_\Psi$ is a real-linear map between real $C^*$-algebras, then it can be extended uniquely to a complex-linear map $\varphi^C: A\longrightarrow B$, that is called the complexification of $\varphi$. Throughout this paper all $C^*$-algebras are complex $C^*$-algebras.

For a $C^*$-algebra $A$ let $\mathcal{M}(A)$ denotes the multiplier algebra of $A$,  where we recall, that $\mathcal{M}(A)$ embeds canonically into the second dual of $A$ i.e. $A^{**}$. In this paper, by c.c.p. and c.p. maps we mean contractive completely  positive and completely  positive maps, respectively.

A $C^*$-algebra $A$ is called amenable in the Johnson sense \cite{j}, if every continuous derivation from $A$ into dual of any Banach $A$-bimodule $X$ is inner.   The $C^*$-algebra $A$ is called nuclear, if for every $\varepsilon>0$ and finite subset $\mathcal{F}$ of $A$, there exist a finite-dimensional $C^*$-algebra $B$ and c.c.p. linear maps $\varphi: A \longrightarrow B$ and   $\psi: B\longrightarrow A$ such that, $\|\psi\circ\varphi(a)-a\|<\varepsilon$ for all $a\in A$. Haagerup showed that these two definitions are equivalent \cite{ha}.  Amenability (nuclearity) of real $C^*$-algebras have been studied recently by Ho \cite{ho}. He obtained equivalent relation between amenability of real $C^*$-algebras and the homomorphisms between these algebras.

In this paper, we investigate the relations between $C^*$-algebras and real part of them (defined by an  involutory $*$-antiautomorphisms). The next section deals with weakly nuclear maps between real $C^*$-algebras and them complexifications. Section three considers exactness of these algebras and the final section studies qusidiagonality of them.
\section{Weakly Nuclear Maps}
Let $A$ and $B$ be two $C^*$-algebras, and let $M$ be a von Neumann algebra. Then, a c.p. map $\varphi:A\longrightarrow\mathcal{M}(B)$ is called weakly nuclear if $b^*\varphi(\cdot)b:A\longrightarrow B$ is nuclear for every $b\in B$. Moreover, a c.p. map $\varphi:A\longrightarrow M$ is called weakly nuclear if it is a point-weak$^*$ limit of maps factoring by c.p. maps through matrix algebras.
\begin{proposition}\label{p1}
Let $A$ and $B$ be two $C^*$-algebras and let $\Phi$ and $\Psi$ be two involutory $*$-antiautomorphisms of $A$ and $B$, respectively. Assume that $\varphi:A_\Phi\longrightarrow B_\Psi$ is a c.p. map such that $\varphi^C:A\longrightarrow B\subseteq B^{**}$ is weakly nuclear, then $\varphi$ is nuclear.
\end{proposition}
\begin{proof}
Since $\varphi$ is a c.p. map, $\varphi^C$ is a c.p. map \cite[Proposition 2]{ho}. Let $\varphi^C:A\longrightarrow B\subseteq B^{**}$ be  weakly nuclear, then $\varphi^C:A\longrightarrow B$ is nuclear \cite[Proposition 2.4 (1)]{ga17}. Then $\varphi$ is nuclear \cite[Proposition 5]{ho}.
\end{proof}
\begin{corollary}\label{c1}
Let $A$ and $B$ be two $C^*$-algebras and let $\Phi$ and $\Psi$ be two involutory $*$-antiautomorphisms of $A$ and $\mathcal{M}(B)$, respectively. Assume that $\varphi:A_\Phi\longrightarrow \mathcal{M}(B)_\Psi$ is a c.p. map such that $\varphi^C:A\longrightarrow \mathcal{M}(B)$ is weakly nuclear, then $\varphi$ is  weakly nuclear.
\end{corollary}
\begin{proof}
Suppose that $\varphi^C:A\longrightarrow \mathcal{M}(B)$ is weakly nuclear, then $\varphi^C:A\longrightarrow \mathcal{M}(B)\subseteq B^{**}$ is weakly nuclear \cite[Proposition 2.4 (2)]{ga17}. This implies that $b^*\varphi^C(\cdot)b:A\longrightarrow B$ is weakly nuclear (as a map into $B^{**}$). Then Proposition \ref{p1} implies that $b^*\varphi(\cdot)b$ is nuclear.
\end{proof}

Let $A$ be a $C^*$-algebra, $H$ be a Hilbert space, $M\subset B(H)$ be a von Neumann algebra and $\varphi:A\longrightarrow M$ be a c.p. map. Then $\varphi$ is weakly nuclear if and only if the product map $\varphi\times i_{M'}:A\otimes_{\max} M'\longrightarrow B(H)$ is continuous with respect to the spatial tensor product \cite[Theorem 3.8.5]{bo}. Note that, the mentioned fact is proven for unital $C^*$-algebras in \cite{bo}, but by normalizing and unitizing the map one can obtain in the stated general case. Now, let $A$, $B$ be two $C^*$-algebra,  $M$ be a von Neumann algebra and $\Phi$, $\Psi$  are involutory $*$-antiautomorphisms of $A$ and $B$, respectively. Let $A$ be nuclear, $\varphi:A_\Phi\longrightarrow M$ and $\psi:B_\Psi\longrightarrow M$ be c.p. maps with commuting images. Since $A$ is nuclear, $A_\Phi$ is too \cite[Proposition 3]{ho} and replacing $A_\Phi\otimes B_\Psi$ by $A$ in the above stated result from \cite{bo}, the c.p. map $\varphi\times(\psi\times i_{M'})$ is a c.p. extensive map of the linear map $(\varphi\times\psi)\times_{alg} i_{M'}:(A_\Phi\otimes B_\Psi)\otimes_{alg}M'\longrightarrow B(H)$. Thus, \cite[Theorem 3.8.5]{bo} implies that the following result:

\begin{corollary}\label{c2}
Let $A$, $B$ be two $C^*$-algebras, $M$ be a von Neumann algebra and let $\Phi$ and $\Psi$ be two involutory $*$-antiautomorphisms of $A$ and $B$, respectively. Suppose that $A$ is nuclear, $\varphi:A_\Phi\longrightarrow M$ and $\psi:B_\Psi\longrightarrow M$ be c.p. maps with commuting images. If $\psi$ is weakly nuclear, then $\phi\times \psi$ is weakly nuclear.
\end{corollary}
Define $\sigma:\mathbb{C}\longrightarrow M_2(\mathbb{R})$ by
\begin{equation}\label{eqt1}
    \sigma(a+ib)=\left(
                   \begin{array}{cc}
                     a & b \\
                    -b & a \\
                   \end{array}
                 \right)
\end{equation}
and $\rho:M_2(\mathbb{R})\longrightarrow\mathbb{C}$ by
\begin{equation}\label{eqt2}
    \rho\left(
                   \begin{array}{cc}
                     a & b \\
                    c & d \\
                   \end{array}
                 \right)=\frac{1}{2}(a+d)+\frac{1}{2}i(b-c).
\end{equation}

Then $\sigma$ and $\rho$ are c.p. maps such that $\rho\circ\sigma$ equal to the identity map \cite[Lemma 4]{ho}.
\begin{theorem}\label{t1}
Let $A$, $B$ be two $C^*$-algebras, and let $\Phi$  be an involutory $*$-antiautomorphisms of $A$. Suppose that  $\varphi:A_\Phi\longrightarrow \mathcal{M}(B)$ is a c.p. map. Then $\varphi$ is weakly nuclear if and only if $\varphi^C$ is weakly nuclear.
\end{theorem}
\begin{proof}
Let $\varphi:A_\Phi\longrightarrow\mathcal{M}(B)$ be weakly nuclear. Thus, $b^*\varphi(\cdot) b:A_\Phi\longrightarrow B$ is nuclear for every $b\in B$. Let $\varepsilon>0$ and $\mathcal{F}=\{a_1+ib_1, a_2+ib_2,\ldots, a_n+ib_n\}$ be a finite subset of $A$. Then $F=\{a_1, a_2,\ldots,a_n,b_1,b_2,\ldots, b_n\}$ is a finite subset of $A_\Phi$. Hence, there are c.c.p. linear maps $\phi:A_\Phi\longrightarrow M_n(\mathbb{R})$ and $\psi:M_n(\mathbb{R})\longrightarrow B$, for some $n\in\mathbb{N}$, such that $\|b^*\varphi(a)b-\psi\circ\phi(a)\|<\varepsilon/2$. Now, by considering the complexifications of $\phi$ and $\psi$ i.e. $\phi^C:A\longrightarrow M_n(\mathbb{C})$ and $\psi^C:M_n(\mathbb{C})\longrightarrow B$ we have the following:
\begin{eqnarray*}
  \|b^*\varphi^C(a_k+ib_k)b-\psi^C\circ\phi^C(a_k+ib_k)\| &\leq&  \|b^*\varphi(a_k)b-\psi\circ\phi(a_k)\|+ \|i\left(b^*\varphi(b_k)b-\psi\circ\phi(b_k)\right)\| \\
   &<&  \varepsilon.
\end{eqnarray*}

This implies that $b^*\varphi^C(\cdot)b:A\longrightarrow B$ is nuclear. Thus, $\varphi^C:A\longrightarrow\mathcal{M}(B)$ is weakly nuclear.

Conversely, suppose that  $\varphi^C:A\longrightarrow\mathcal{M}(B)$ is weakly nuclear. This implies that $b^*\varphi^C(\cdot)b:A\longrightarrow B$ is nuclear. Then,  for every $\varepsilon>0$ and finite subset $F$ of $A_\Phi$ there are c.c.p. linear maps $\phi:A\longrightarrow M_n(\mathbb{C})$ and $\psi:M_n(\mathbb{C})\longrightarrow B$, for some $n\in\mathbb{N}$, such that
\begin{equation}\label{eqtt}
    \|b^*\varphi^C(a)b-\psi\circ\phi(a)\|<\varepsilon,
\end{equation}
 for every $a\in F$. Define $\phi':A\longrightarrow M_{2n}(\mathbb{R})$ by $\phi'=\sigma^{(n)}\circ\phi$ and $\psi':M_{2n}(\mathbb{R})\longrightarrow B$ by $\psi'=\psi\circ\rho^{(n)}$ for $n\in\mathbb{N}$, where $\sigma$ and $\rho$ are c.p. maps defined in \eqref{eqt1} and \eqref{eqt2}. Then
\begin{equation}\label{eqt3}
    \psi'\circ\phi'=\psi\circ\phi.
\end{equation}

Thus, \eqref{eqtt} together with \eqref{eqt3} implies
\begin{equation}\label{eqt4}
    \|b^*\varphi(a)b-\psi'\circ\phi'(a)\|=\|b^*\varphi(a)b-\psi\circ\phi(a)\|<\varepsilon,
\end{equation}
 for every $a\in F$. Hence, $b^*\varphi(a)b:A_\Phi\longrightarrow\mathcal{M}(B)$ is nuclear. Therefore, $\varphi$ is weakly nuclear.
 \end{proof}

 \section{Exactness}
 Let $A$, $B$ be two $C^*$-algebras and let $A\otimes_{\text{alg}} B$ be algebraic tensor product between them. If $\Phi$ is an involutory $*$-antiautomorphism of $A$, then one can see that
 \begin{equation}\label{eq1s3}
    A\otimes_{\text{alg}} B=(A_\Phi\otimes_{\text{alg}} B)+(iA_\Phi\otimes_{\text{alg}} B).
 \end{equation}

We denote the completion of $A\otimes_{\text{alg}} B$ with the spatial (or minimal $C^*$-norm) norm $\|\cdot\|_{\min}$ by $A\otimes_{\min} B$. Thus, we write \eqref{eq1s3} as follows
 \begin{equation}\label{eq2s3}
    A\otimes_{{\min}} B=(A_\Phi\otimes_{{\min}} B)+(iA_\Phi\otimes_{{\min}} B).
 \end{equation}

 A $C^*$-algebra $A$ is said to be exact if the functor $A\otimes_{\min}-$ is exact; i.e., if every short exact sequence,
\begin{equation}\label{exact1}
    0\longrightarrow I\stackrel{\imath}{\longrightarrow}B\stackrel{\pi}{\longrightarrow}\frac{B}{I}\longrightarrow0
\end{equation}
the sequence
\begin{equation}\label{exact2}
  0\longrightarrow A\otimes_{\min}I\stackrel{id_A\otimes\imath}{\longrightarrow}A\otimes_{\min}B\stackrel{id_A\otimes\pi}{\longrightarrow}A\otimes_{\min}\frac{B}{I}\longrightarrow0
\end{equation}
is also exact, where $I$ is a closed ideal of $B$. An example of exact $C^*$-algebras is nuclear $C^*$-algebras \cite[Theorem 6.5.2]{morb}. Let $A$ be a $C^*$-algebra  and $\Phi$ be an involutory $*$-antiautomorphism of $A$. We say that $A_\Phi$ is exact if every short exact sequence of  $C^*$-algebras,
\begin{equation}\label{exact3}
    0\longrightarrow I\stackrel{\imath}{\longrightarrow}B\stackrel{\pi}{\longrightarrow}\frac{B}{I}\longrightarrow0
\end{equation}
the sequence
\begin{equation}\label{exact4}
  0\longrightarrow A_\Phi\otimes_{\min}I\stackrel{id_{A_\Phi}\otimes\imath}{\longrightarrow}
  A_\Phi\otimes_{\min}B\stackrel{id_{A_\Phi}\otimes\pi}{\longrightarrow}A_\Phi\otimes_{\min}\frac{B}{I}\longrightarrow0
\end{equation}
is also exact, where $I$ is a closed ideal of $B$. Similarly, if $A_\Phi$ is nuclear, then it is exact.

Let $\varphi:A_\Phi\longrightarrow\mathbb{C}$ be a linear bounded map. Following \cite[Theorem 1]{tom} one can define a  continuous linear map $R_\varphi:A_\Phi\otimes_{\min}B\longrightarrow B$ by $R_\varphi(a\otimes b)=\varphi(a)b$ for every $a\in A_\Phi$ and $b\in B$. Similarly, if $\psi:B\longrightarrow\mathbb{R}$ is a  bounded  linear map, then there is a continuous linear  map $L_\psi:A_\Phi\otimes_{\min}B\longrightarrow A_\Phi$ by $R_\psi(a\otimes b)=\psi(b)a$ for every $a\in A_\Phi$ and $b\in B$. Suppose that $A_1$ and $B_1$ are  $C^*$-subalgebras of $A_\Phi$ and $B$, respectively. For every bounded linear mappings $\varphi:A_\Phi\longrightarrow\mathbb{C}$ (is real-linear) and $\psi:B \longrightarrow\mathbb{R}$, the Fubini product of $A_1$ and $B_1$ with respect to $A_\Phi$ and $B_\Psi$ is
\begin{equation}\label{fp}
    F(A_1, B_1, A_\Phi\otimes B)=\{x\in  A_\Phi\otimes B :\ R_\varphi(x)\in B_1,\ L_\psi(x)\in A_1\}.
\end{equation}

Moreover, for every non-zero $a\otimes b\in A_\Phi\otimes B$, $\varphi\otimes \psi(a\otimes b)=\varphi(a)\psi(b)$ is a bounded linear map on $A_\Phi\otimes B$ and it is easy to see that the set of such $\varphi\otimes \psi$ mappings separates  $A_\Phi\otimes B$. Also, from \cite[Theorem 1]{tom} we have $\varphi\otimes\psi(x)=\psi\left(R_\varphi(x)\right)=\varphi\left(L_\psi(x)\right)$. This shows that both $R_\varphi$ and $L_\psi$ separate $A_\Phi\otimes B$ for every $\varphi$ and $\psi$ defined as above.
\begin{theorem}
Let $A$ be a $C^*$-algebra and $\Phi$ be an involutory $*$-antiautomorphism of $A$. Then $A$ is exact if and only if $A_\Phi$ is exact.
\end{theorem}
\begin{proof}
Let $A$ be an exact $C^*$-algebra. Suppose that $B$ is an arbitrary $C^*$-algebra and $I$ is a closed ideal in $B$, where they are satisfied in \eqref{exact3}.   We must show that the sequence \eqref{exact4} is exact. For this, it is sufficient that we show $ F(A_\Phi, I, A_\Phi\otimes B) =A_\Phi\otimes_{\min}I$ \cite[Theorem 1.1]{ki}.

For every linear bounded map $\varphi:A_\Phi\longrightarrow\mathbb{C}$, we denote its complexification by $\varphi^C$. Then $\varphi^C|_{A_\Phi}=\varphi$ and $R_{\varphi}(x)=R_{\varphi^C}(x)$ for every $x\in A_\Phi\otimes B_\Psi$. Moreover,
 \begin{eqnarray}\label{exact5}
 \nonumber
    R_{\varphi^C}\circ(id_A\otimes\pi)(a\otimes b)&=&R_{\varphi^C}(a\otimes\pi(b))=\varphi^C(a)\pi(b)\\
    &=&\pi\circ R_{\varphi^C}(a\otimes b),
 \end{eqnarray}
for every $a\otimes b\in A\otimes_{\min} B$. Since $\varphi$ separates $A_\Phi$, $\varphi^C$ separates $A$ and consequently, \eqref{exact5} implies that $R_{\varphi^C}$ separates $A\otimes_{\min}B$. Consider the exact sequence
\begin{equation}\label{exact6}
  0\longrightarrow A\otimes_{\min}I^C\stackrel{id_{A}\otimes\imath^C}{\longrightarrow}A\otimes_{\min}B
  \stackrel{id_{A}\otimes\pi}{\longrightarrow}A\otimes_{\min}\frac{B}{I}\longrightarrow0
\end{equation}

 Since $R_{\varphi^C}$ separates $A\otimes_{\min}B$, $\ker (id_{A}\otimes\pi)=F(A, I, A\otimes_{\min}B)$. Exactness of $A$ implies that $F(A, I, A\otimes_{\min}B)=A\otimes_{\min}I$. Let $x\in F(A_\Phi, I, A_\Phi\otimes B)$. Then $R_\varphi(x)=R_{\varphi^C}(x)\in I$. Thus, $x\in A\otimes_{\min}I$. This means that
 \begin{equation}\label{exact7}
    x\in (A\otimes_{\min}I)\cap(A_\Phi\otimes B)=A_\Phi \otimes_{\min}I.
 \end{equation}

 This implies that $A_\Phi$ is exact.

 Conversely, assume that $A_\Phi$ is exact. Thus, $iA_\phi$ is exact, too. We must show that $\ker(id_A\otimes \pi)=A\otimes I$, where $id_A\otimes \pi:A\otimes_{\min} B\longrightarrow A\otimes_{\min}\frac{B}{I}$. Set $A_1=A_\Phi$, $A_2=iA_\Phi$ and denote $A_1+A_2=\coprod_{i=1}^2A_i$. Since $A_i$'s for $i=1,2$ are exact, $\ker(id_{A_i}\otimes \pi)=A_i\otimes I$.  Then the diagram\\
\begin{equation*}
\begin{CD}
A\otimes_{\min} B @> \ \ \ \ id_A\otimes \pi\ \ \ >> A\otimes_{\min}\frac{B}{I}  \\
f @VV V @VV gV\\
 \coprod_{i=1}^2 (A_i\otimes_{\min}B) @> \coprod_{i=1}^2(id_{A_i}\otimes \pi)>> \coprod_{i=1}^2 (A_i\otimes_{\min}\frac{B}{I})
\end{CD}
\end{equation*}
commutes. Indeed $f$ and $g$ are the identity maps, hence,
\begin{eqnarray}\label{exact8}
\nonumber
    \ker(id_A\otimes \pi)&=&\ker\coprod_{i=1}^2(id_{A_i}\otimes \pi)=\coprod_{i=1}^2(A_i\otimes_{\min}I)\\
    &=&A\otimes_{\min}I.
\end{eqnarray}

This implies that $A$ is exact.
\end{proof}

By the above result and \cite[Proposition 6.1.10]{ro} we have the following results.
\begin{corollary}
Let $A$ be a $C^*$-algebra and $\Phi$ be an involutory $*$-antiautomorphism of $A$. If $A_\Phi$ is  exact  and if $G$ is an amenable, locally compact group acting on $A$, then the crossed product $A \rtimes G$ is exact.
\end{corollary}
\begin{corollary}
Let $A$, $B$ be two $C^*$-algebras and $\Phi$, $\Psi$ be  involutory $*$-antiautomorphisms of $A$ and $B$, respectively. Then $A_\Phi$ and $B_\Psi$ are  exact   if and only if $A_\Phi \otimes_{\min} B_\Psi$ is exact.
\end{corollary}
\section{Quasidiagonality}
A  unital $C^*$-algebra $A$ is called quasidiagonal, if, for every finite subset $\mathcal{F}$ of $A$ and $\varepsilon>0$, there exist a matrix algebra $M_k(\mathbb{C})$ and a unital c.p. linear map $\varphi:A\longrightarrow M_k(\mathbb{C})$ such that
\begin{equation}\label{eq1qd}
    \|\varphi(ab)-\varphi(a)\varphi(b)\|<\varepsilon,\hspace{1cm} a,b\in \mathcal{F},
\end{equation}
and
\begin{equation}\label{eq2qd}
   |\ \|\varphi(a)\|-\|a\||<\varepsilon,\hspace{2.5cm} a\in \mathcal{F}.
\end{equation}

Let $A$ be a $C^*$-algebra and $\Phi$ be an involutory $*$-antiautomorphism of $A$. Then for every $c=a+ib\in A$, we consider $\|c\|=\|a\|+\|b\|$ and for considering  the qusidiagonality of $A_\Phi$ we replace $M_k(\mathbb{C})$ by $M_k(\mathbb{R})$. We work with the $\|\cdot\|_1$ on  $M_k(\mathbb{R})$.

\begin{theorem}\label{mtqd}
Let $A$ be a $C^*$-algebra and $\Phi$ be an involutory $*$-antiautomorphism of $A$. Then $A$ is quasidiagonal if and only if $A_\Phi$ is quasidiagonal.
\end{theorem}
\begin{proof}
Let $A_\Phi$ be quasidiagonal and let $\mathcal{F}_A=\{c_1=a_1+ib_1, c_2=a_2+ib_2,\ldots, c_n=a_n+ib_n\}$ be a finite subset in $A$. Set $\mathcal{F}=\{a_1, a_2,\ldots, a_n,b_1, b_2,\ldots,b_n\}$ is a finite subset in $A_\Phi$. For every $\varepsilon>0$, there exist a matrix algebra $M_k(\mathbb{R})$ and a unital c.p. real-linear map $\varphi:A_\Phi\longrightarrow M_k(\mathbb{R})$ such that
\begin{equation}\label{eq1qd}
    \|\varphi(ab)-\varphi(a)\varphi(b)\|<\frac{\varepsilon}{4},\hspace{1cm} a,b\in \mathcal{F},
\end{equation}
and
\begin{equation}\label{eq2qd}
    \left|~ \|\varphi(a)\|-\|a\|\right|<\frac{\varepsilon}{4},\hspace{2.5cm} a\in \mathcal{F}.
\end{equation}

Suppose that $\varphi^C:A\longrightarrow M_k(\mathbb{C})$ is the complexification of $\varphi$ that is a c.p. map. Then
\begin{eqnarray*}
  \|\varphi^C(c_kc_l)-\varphi(c_k)^C\varphi^C(c_l)\| &=&\|\varphi^C\left([a_ka_l-b_kb_l]+i[b_ka_l+a_kb_l]\right)-\varphi^C(a_k+ib_k)\varphi^C(a_k+ib_k)\|  \\
   &\leq& \|\varphi(a_ka_l-b_kb_l)-\varphi(a_k)\varphi(a_l)+\varphi(b_k)\varphi(b_l)\|\\
   &&\hspace{-0.3cm} +\|i[\varphi(b_ka_l+a_kb_l)- \varphi(b_k)\varphi(a_l)-\varphi(a_k)\varphi(b_l)\|  \\
   &<& \frac{\varepsilon}{4}+\frac{\varepsilon}{4}+\frac{\varepsilon}{4}+\frac{\varepsilon}{4} \\
   &=&\varepsilon,
\end{eqnarray*}
for every $c_k, c_l\in \mathcal{F}_A$. Also,
\begin{eqnarray*}
 |~ \|\varphi^C(c_k)\|-\|c_k\| ~|&=&|~\|\varphi^C(a_k+ib_k)\|-\|(a_k+ib_k)\|~|  \\
   &=&|~ \|\varphi(a_k)\|-\|a_k\|+\|\varphi(b_k)\|-\|b_k\|~|\\
   &<& \frac{\varepsilon}{4}+\frac{\varepsilon}{4} \\
   &<&\varepsilon,
\end{eqnarray*}
for every $c_k\in \mathcal{F}_A$.

Let $\sigma:\mathbb{C}\longrightarrow M_2(\mathbb{R})$ be as \eqref{eqt1}. It easy to check that $\sigma$ and $\sigma^{(k)}:M_k(\mathbb{C})\longrightarrow M_{2k}(\mathbb{R})$ for some $k\in\mathbb{N}$ are c.c.p. homomorphisms. Now, define $\theta^{(k)}:M_k(\mathbb{C})\longrightarrow M_{2k}(\mathbb{R})$ by
$$\theta^{(k)}([a_{jl}+ib_{jl}])=\frac{1}{\max_{1\leq l\leq k}\sum_{j=1}^k(|a_{jl}|+|b_{jl}|)+1}\sigma^{(k)}([a_{jl}+ib_{jl}]),$$
 for $[a_{jl}+ib_{jl}]\in M_k(\mathbb{C})$ and $1\leq j,l\leq k$. Then $\theta^{(k)}$ is a c.c.c.p.  homomorphism.
  Assume that $A$ is quasidiagonal. Thus, for every finite subset $\mathcal{F}$ of $A_\Phi$, $\varepsilon>0$, there exist matrix algebra $M_k(\mathbb{C})$ and c.p. linear map $\varphi:A\longrightarrow M_k(\mathbb{C})$ such that \eqref{eq1qd} and \eqref{eq2qd} hold for every $a\in\mathcal{F}$. Define $\varphi':A\longrightarrow M_{2k}(\mathbb{R})$ by $\varphi'=\theta^{(k)}\circ\varphi\circ\Phi\circ*$. Clearly, $\Phi\circ*$ is a real-linear homomorphism and $\theta^{(k)}$ is real-linear, so $\varphi'$ is a c.c.p. real-linear map. Then
\begin{eqnarray*}
  \|\varphi'(ab)-\varphi'(a)\varphi'(b)\| &=&\|\theta^{(k)}\circ\varphi\circ\Phi\circ*(ab)-\theta^{(k)}\circ\varphi\circ\Phi\circ*(a)\ \theta^{(k)}\circ\varphi\circ\Phi\circ*(b)\|  \\
   &=& \|\theta^{(k)}\circ\varphi(ab)-\theta^{(k)}\circ\varphi(a)\ \theta^{(k)}\circ\varphi(b)\|\\
   &=& \|\theta^{(k)}\left(\varphi(ab)-\varphi(a)\varphi(b)\right)\| \\
   &\leq&\|\theta^{(k)}\|\ \|\varphi(ab)-\varphi(a)\varphi(b)\|\\
   &<&\varepsilon,
\end{eqnarray*}
for every $a,b\in \mathcal{F}$. Also,
\begin{eqnarray*}
  |~\|\varphi'(a)\|-\|a\|~| &=&|~\|\theta^{(k)}\circ\varphi\circ\Phi\circ*(a)\|-\|a\|~| =|~\|\theta^{(k)}\circ\varphi(a)\|-\|a\|~| \\
   &\leq& |~\|\theta^{(k)}\|~\|\varphi(a)\|-\|a\|~|\leq|~\|\varphi(a)\|-\|a\|~|\\
   &<& \varepsilon,
\end{eqnarray*}
for every $a\in \mathcal{F}$.
\end{proof}

Let $\mathcal{Q}$ be the universal UHF algebra, $\omega$ be a free ultrafilter on $\mathbb{N}$ and let $\mathcal{Q}_\omega$ be the ultrapower of $\mathcal{Q}$ defined by
\begin{equation}\label{up}
    \mathcal{Q}_\omega:=\ell^\infty(\mathcal{Q})/\{(a_n)\in\ell^\infty(\mathcal{Q})\ :\ \lim_{n\to\infty}\|a_n\|=0\}.
\end{equation}

The above Theorem and \cite[Proposition 1.4]{tww} imply the following result:
\begin{corollary}\label{fc}
Let $A$ be a separable, unital and nuclear $C^*$-algebra and $\Phi$ be an involutory $*$-antiautomorphism of $A$. Then the following statements are equivalent:
\begin{itemize}
  \item[(i)] $A_\Phi$ is quasidiagonal;
  \item[(ii)] there exists a unital embedding $A\hookrightarrow \mathcal{Q}_\omega$.
\end{itemize}
\end{corollary}
\begin{corollary}\cite[Theorem 2.2]{or}\label{fc1}
Let $A$ be a separable, unital  $C^*$-algebra, $\Phi$ be an involutory $*$-antiautomorphism of $A$ and let $G$ be a discrete countable amenable and residually finite group with a sequence of F{\o}lner sets $F_n$ and tilings of the form $G = K_nL_n$ with $F_n \subset K_n$ for all $n\in \mathbb{N}$. Let $\alpha:G\longrightarrow \text{Aut}(A)$ be a homomorphism such that
$$\lim_{n\to\infty}\left(\max_{g\in L_n\cap K_nK_n^{-1}F_n}\|\alpha(g)a-a\|\right)=0.$$

If $A_\Phi$ is quasidiagonal, then $A\rtimes_\alpha G$ is quasidiagonal.
\end{corollary}
Let $A$ be a $C^*$-algebra, $\Phi$ be an involutory $*$-antiautomorphism of $A$ and $\tau:A_\Phi\longrightarrow\mathbb{R}$ be a real-linear positive functional such that $\tau(ab)=\tau(ba)$ for every $a,b\in A_\Phi$. Then $\tau$ extends uniquely to a complex-linear positive map $\tau^C:A\longrightarrow\mathbb{C}$ such that $\tau(ab)=\tau(ba)$ for every $a,b\in A$. The complex-linear positive map $\tau^C$ is called tracial functional. If $A$ is unital, then $\tau^C$ is called quasidiagonal if for every $\mathcal{F}$ of $A$ and $\varepsilon>0$, there exists a matrix algebra $M_k(\mathbb{C})$ and a unital c.p. linear map $\varphi:A\longrightarrow M_k(\mathbb{C})$ such that
\begin{equation}\label{eq1qd}
   \|\varphi(ab)-\varphi(a)\varphi(b)\|<\varepsilon,\hspace{1cm} a,b\in \mathcal{F},
\end{equation}
and
\begin{equation}\label{eq2qd}
    |\tau_{M_k}\circ\varphi(a)-\tau^C(a)|<\varepsilon,\hspace{2.5cm} a\in \mathcal{F},
\end{equation}
where $\tau_{M_k}$ is the unique normalized trace on $M_k(\mathbb{C})$.

\begin{lemma}\label{lem1}
Let $\tau_{M_k}$ and $\tau_{M_{2k}}'$ be the unique normalized traces on $M_k(\mathbb{C})$ and $M_{2k}(\mathbb{R})$ for $k\in\mathbb{N}$, respectively. Then there are c.p. maps $\eta:\mathbb{C}\longrightarrow M_2(\mathbb{R})$ and $\upsilon:\mathbb{C}\longrightarrow\mathbb{R}$ such that $\|\eta\|\leq1$ and
\begin{equation}\label{eqtr1}
   \upsilon\circ\tau_{M_k}=\tau_{M_{2k}}'\circ\eta^{(k)}.
\end{equation}
\end{lemma}
\begin{proof}
Let $\eta:\mathbb{C}\longrightarrow M_2(\mathbb{R})$ defined by
\begin{equation}\label{eqtr2}
    \eta(a+ib)=\left(
                 \begin{array}{cc}
                  a  & 0 \\
                  0  & b \\
                 \end{array}
               \right),
\end{equation}
and $\upsilon:\mathbb{C}\longrightarrow\mathbb{R}$ by $\upsilon(a+ib)=a+b$. These real-linear maps are c.p maps. Clearly, \eqref{eqtr1} holds.
\end{proof}
\begin{proposition}\label{p4}
Let $A$ be a  $C^*$-algebra, $\Phi$ be an involutory $*$-antiautomorphism of $A$ and $\tau$ be a tracial functional on $A$. If $\tau$ is quasidiagonal, then there is a  tracial functional $\tau_{A_\Phi}$ on $A_\Phi$ such that is quasidiagonal. Moreover, if $A$ is separable, unital, exact and $\tau_{A_\Phi}$ is faithful, then $A$ is quasidiagonal.
\end{proposition}
\begin{proof}
Let $\tau:A\longrightarrow\mathbb{C}$ be a quasisdiagonal trace, $\mathcal{F}\subset A_\Phi$ be a finite subset, \eqref{eq1qd} and \eqref{eq2qd} hold for every $\varepsilon>0$ and $a,b\in\mathcal{F}$. Define the mappings $\sigma$, $\theta$  and $\varphi'$ as in the proof of the Theorem \ref{mtqd}, then we have
$$ \|\varphi'(ab)-\varphi'(a)\varphi'(b)\|<\varepsilon.$$

Suppose that $\eta$ and $\upsilon$ are the obtained c.p. maps in the Lemma \ref{lem1}. Define $\eta_1(a+ib)=\eta(a+i|b|)$ and $\upsilon_1(a+ib)=\upsilon(a+|b|)$. Note that by these definitions \eqref{eqtr1} holds for $\eta_1$ and $\upsilon_1$. Clearly, $\upsilon_1\circ\tau=\tau_{A_{\Phi}}$ is a trace on $A_\Phi$. Also,
\begin{equation}\label{eq1t2}
\tau_{M_{2k}}'\circ\theta^{(k)}(\mathfrak{a}) \leq \tau_{M_{2k}}'\circ\eta_1^{(k)} (\mathfrak{a}),
\end{equation}
for every $\mathfrak{a} \in M_{k}(\mathbb{C})$. Then
\begin{eqnarray*}
  |\tau_{M_{2k}}\circ\varphi'(a)-\tau_{A_{\Phi}}(a)| &=&|\tau_{M_{2k}}\circ \theta^{(k)}\circ\varphi\circ\Phi\circ*(a)-\tau_{A_{\Phi}}(a)| =|\tau_{M_{2k}}\circ \theta^{(k)}\circ\varphi(a)-\tau_{A_{\Phi}}(a)|  \\
   &\leq& |\tau_{M_{2k}}'\circ\eta_1^{(k)}\circ\varphi(a)-\upsilon_1\circ\tau(a)| \\
   &=& | \upsilon_1\circ\tau_{M_k}\circ\varphi(a)-\upsilon_1\circ\tau(a)|\\
   &<& \|\upsilon_1\|\varepsilon,
\end{eqnarray*}
for every $a\in \mathcal{F}$. Thus $\tau_{A_\Phi}$ is quasisdiagonal. If $\tau_{A_\Phi}$ is faithful, then $A_\Phi$ is quasisdiagonal \cite[Proposition 3.4]{ga17}. Then Theorem \ref{mtqd} implies that $A$ is quasisdiagonal.
\end{proof}

\end{document}